\documentclass[11pt,reqno]{amsart}
\usepackage{amssymb, amsmath, amsthm,amsrefs,graphicx,marvosym}
\usepackage{latexsym}
\usepackage{color}
\usepackage{exscale}
\usepackage{bm, bbm, mathrsfs}
\usepackage{hyperref}

\newtheorem{theorem}{Theorem}[section]

\newtheorem{definition}[theorem]{Definition}
\newtheorem{remark}[theorem]{Remark}
\newtheorem{example}[theorem]{Example}

\newtheorem{claim}[theorem]{Claim}

\newtheorem{convention}{Convention}

\newcommand{\RR}{\mathbb R}

\def\vp{\varphi}

\def\s0{{ s_0}}

\def\ts0{{\tilde s_0}}

\def\eq#1{(\ref{#1})}
\def\nn{\nonumber}

\def\({\left(\begin{array}{cccccc}}
\def\){\end{array}\right)}

\def\bes{\begin{eqnarray}}
\def\ees{\end{eqnarray}}

\newcommand{\del}{\partial}

\newcommand{\beq}{\begin{equation}}
\newcommand{\eeq}{\end{equation}}
\newcommand{\bea}{\begin{eqnarray}}
\newcommand{\eea}{\end{eqnarray}}
\newcommand{\baln}{\begin{align}}
\newcommand{\ealn}{\end{align}}
\newcommand{\beann}{\begin{eqnarray*}}
\newcommand{\eeann}{\end{eqnarray*}}
\newcommand{\si}{\ensuremath{\sigma}}

\newcommand{\minus}{\smallsetminus}
\newcommand{\ul}{\underline}

\newcommand{\I}{\ensuremath{\mathcal{I}}}

\newcommand{\pf}{\begin{proof}}
\newcommand{\foorp}{\end{proof}}

\begin{document}

\title{On Two Theorems of Darboux}
\author{Michael Benfield}\address{M.~Benfield, San Diego, CA, \tt{mike.benfield@gmail.com}.}\thanks{M.~Benfield was partially supported by NSF grant
DMS-1311743.}
\author{Helge Kristian Jenssen }\address{ H.~K.~Jenssen, Department of
Mathematics, Penn State University, {\tt
jenssen@math.psu.edu}.}\thanks{ H.~K.~Jenssen was partially supported by NSF grant
DMS-1311353.}
\author{Irina A.\ Kogan}\address{I.~A.~Kogan, Department of Mathematics, North Carolina State University,
     {\tt iakogan@ncsu.edu}.}\thanks{ I.~A.~Kogan was partially supported by NSF grant
DMS-1311743.}
\date{\today}

\begin{abstract}
We provide precise formulations and proofs of two theorems 
from Darboux's lectures on orthogonal systems
\cite{dar}. These results provide local existence and uniqueness of 
solutions to certain types of first order PDE systems where  
each equation contains a single derivative for which it is solved:
\[\frac{\partial u_i}{\partial x_j}(x)=f_{ij}(x,u(x)).\] 
The data prescribe values for the unknowns $u_i$ along certain 
hyperplanes through a given point $\bar x$.

The first theorem applies to determined systems (the number of 
equations equals the number unknowns), and a unique, local 
solution is obtained via Picard iteration. While Darboux's statement 
of the theorem leaves unspecified ``certaines conditions de 
continuit\'e,'' it is clear from his proof that he assumes Lipschitz 
continuity of the maps $f_{ij}$. On the other hand, he did not 
address the regularity of the data. 
We provide a precise formulation and proof of his first theorem.

The second theorem is more involved and applies to overdetermined 
systems of the same general form. Under the appropriate integrability 
conditions, Darboux used his first theorem to treat the cases with 
two and three independent variables.
We provide a proof for any number of independent variables.

While the systems are rather special, they do appear in applications; 
e.g., the second theorem contains the classical Frobenius theorem 
on overdetermined systems as a special case.
The key aspect of the proofs is that they apply to non-analytic 
situations. In an analytic setup the results are covered by the general 
Cartan-K\"ahler theorem.
\end{abstract}

\maketitle
\noindent {\bf Keywords:} Overdetermined systems of PDEs; integrability theorems. 

\noindent {\bf MSC 2010:} 35N10.

\tableofcontents

\section{Introduction}\label{intro}
Darboux, in Chapter I of Livre III in his monograph 
``Syst\`emes Orthogonaux'' \cite{dar}, stated three integrability 
theorems for certain types of first order systems of PDEs.
The second one of these is the classical Frobenius theorem and will 
not be considered in detail in what follows (see Remark \ref{frob_rmk}); 
the remaining ones are his ``Th\'eor\`eme I'' and ``Th\'eor\`eme III'' and
these are the ``two theorems'' in our title.

The theorems give local existence for certain types of first order systems 
of PDEs of the general form
\beq\label{gen_form1}
	u_{i,x_j}(x)=f_{ij}(x,u(x)),
\eeq
where $u$ denotes the vector of unknown functions $u_1,\dots,u_N$, 
the independent variables $x=(x_1,\dots,x_n)$ ranges over an open
set about a fixed point $\bar x\in \RR^n$, and the $f_{ij}$ are given 
scalar functions. 
The data that the systems under consideration are to take on consist 
of given functions prescribed along (typically several) hyperplanes 
through the point $\bar x\in \RR^n$. Darboux makes repeated use of  
Th\'eor\`eme I in his proof of Th\'eor\`eme III.

Concerning regularity, we assume throughout that the functions $f_{ij}$
as well as the assigned data functions are (at least) continuous.
By a ``solution'' we always mean a classical, $C^1$ function $u$ 
which satisfy the PDEs and the data at every point of some 
neighborhood of $\bar x$. 

We note the special character of the equations: each equation 
contains a single derivative with respect to which it is solved. 
As commented by Spivak in \cite{spi5} in connection with the PDE 
formulation of the Frobenius theorem, it is rather laughable to call 
these PDEs at all as they do not relate different partial derivatives 
to each other. Yet, these are basic equations and they are useful. 
Our motivation for revisiting Darboux's results has been their 
application to PDE systems that appear when asking for 
the existence of maps $g:\RR^n\to\RR^n$ whose Jacobian matrix 
has a set of prescribed eigenvector fields \cite{jk1}.

\subsection{Determined systems}\label{det_syst}
The first theorem concerns determined systems, i.e.\ the number $N$ of 
dependent variables $u_i$ equals the number of equations. The data in this
case consists of $N$ prescribed functions along certain hyperplanes through
$\bar x$ (see Section \ref{D_1st_thm} for details).

Darboux employed Picard iteration to establish local existence 
of a solution, and the proof is quite similar to that of the classical Picard-Lindel\"of 
theorem for local existence to Cauchy problems for ODEs \cite{hart}. 
Nevertheless, some care is required to obtain a precise result which is 
explicit about the regularity of the data and the functions $f_{ij}$. In Darboux's
proof the first step is to apply a change of coordinates $x\mapsto y$ and 
$u\mapsto v$ so as to obtain an equivalent PDE system
\[(v_i)_{y_j}(y)=F_{ij}(y,v(y)).\]
The change of variables is made so that the point $\bar x$ corresponds to the origin and 
the data for the $v$-variables vanish identically. As is made explicit in his proof, Darboux 
assumes Lipschitz continuity of the maps $F_{ij}$ (with respect to the dependent 
variables $v$). Next, Darboux sets up the appropriate Picard iteration and 
establishes existence of a local solution $v$. Finally, he establishes uniqueness
by exploiting (in a now-standard fashion) the assumed Lipschitz continuity of the $F_{ij}$.
With that his proof is finished.

However, the following ``detail'' was not addressed by Darboux: the change of 
variables depends on the originally prescribed data, and no regularity 
assumptions are imposed on the data. 
Our first task is to address this issue and provide local existence under 
precise regularity conditions on the original right hand sides $f_{ij}$, as well as on 
the original $u$-data. Different setups are possible; we opt for Lipschitz 
continuity of the $f_{ij}$ (with respect to $u$) and just continuity of the $u$-data. 
The precise statement is given in Theorem \ref{darb_1} below. 

We note two points about this result. First, it provides an easy, ``ODE method'' 
proof for local solutions to a certain type of PDE systems under mild regularity conditions.
Even in an analytic setting, the standard Cauchy-Kowalevskaya 
theorem does not apply as the data are prescribed along several hyperplanes.
(The only exception to this last statement is the case when the given system 
\eq{gen_form1} degenerates to an ODE system, i.e.\ the independent variable 
$x_i$ in \eq{gen_form1} is the same in all equations.)

Concerning the regularity setup we use, other choices are possible.
E.g., it would be of interest to formulate an existence result for 
Carath\'eodory-type solutions for systems \eq{gen_form1}. However, we will not 
pursue this in the present paper.

\subsection{Overdetermined systems}
The second theorem (Darboux's Th\'eor\`eme III) concerns overdetermined 
systems of the same general form \eq{gen_form1}. Under the appropriate 
integrability conditions, which are rather restrictive, it guarantees 
existence of a unique solution near a point $\bar x$, again for data given 
along hyperplanes through $\bar x$.
This result is more challenging to prove, and Darboux limited himself to the 
cases with two or three independent variables. It is not immediate to 
generalize his argument to cases with more independent variables. 
We shall provide a proof by induction on the number of independent 
variables.

Concerning regularity assumptions, we note that a key ingredient in 
both Darboux's and our proof for the second theorem, is the application 
of the first theorem. In particular, it will be applied to the 1st order system 
satisfied by quantities of the form ``LHS\eq{gen_form1} $-$ RHS\eq{gen_form1}.'' 
We will therefore need to require more regularity than for the 
first theorem. Another reason for increased regularity is to make the 
integrability conditions valid.

We note that the second theorem is a consequence of the much
more general Cartan-K\"ahler theorem. However, that theorem depends
of the Cauchy-Kowalevskaya theorem and only applies to the much more 
restrictive setting of analytic systems with analytic data. 
For the precious few results on overdetermined systems which 
do not require analyticity, see \cite{yang} and references therein. 
The work \cite{yang} is formulated in a smooth ($C^\infty$) setting; the proofs 
are not easy. In contrast, the present work employs entirely elementary tools.

\section{Darboux's theorem on determined systems}\label{D_1st_thm}
\subsection{Setup and notation}\label{D_1st_thm_statement}
We consider determined systems of PDEs with the following structure:
\begin{itemize}
\item[(i)] there are $N$ first order PDEs for $N$ unknown functions,
\item[(ii)] each equation contains exactly one derivative and is assumed solved 
for this derivative, and
\item[(iii)] each unknown appears differentiated exactly once.
\end{itemize} 
There is no constraint regarding which first derivatives appear in the equations.
E.g., they could all be with respect to the same independent variable (giving an 
ODE system with parameters) or with respect to some or all of the independent 
variables.

We denote the independent variables by $x=(x_1,\dots,x_n)$ and the dependent 
variables by $u=(u_1,\dots,u_N)\in \RR^N$. For each $i$, $1\leq i\leq n$,
let $u_{i1},\dots,u_{iN_i}$ ($0\leq N_i\leq N$) denote the 
dependent variables that appear differentiated with respect to $x_i$ in the system.
Note that we allow for the possibility that $N_i=0$, i.e.\ none of the equations 
involve differentiation with respect to $x_i$, in which case $x_i$ plays the 
role of a parameter. 
According to assumption (iii) we have that 
\beq\label{N}
	N=\sum_{i=1}^n N_i,
\eeq
and we relabel coordinates so as to write $u$ in the form
\beq\label{u_compns}
	u=(u_1,\dots,u_N)=(u_{11},\dots,u_{1N_1},u_{21},\dots,u_{2N_2},
	\dots,u_{n1},\dots,u_{nN_n}). 
\eeq
The PDEs involving differentiation with respect to $x_i$ then appear consecutively,
and we write them as
\beq\label{sys_i}
	\frac{\del u_{ih}}{\del x_i} = f_{ih}(x_1,\dots,x_n,u_1,\dots,u_N).
\eeq
Here and below it is understood that the index $i$ ranges between $1$ 
and $n$, and that for each such $i$ the index $h$ ranges between $1$
and $N_i$ (unless $N_i=0$).

We next describe the data for the system. For a given point $\bar x\in \RR^n$ and for 
each $i$, we prescribe $N_i$ functions $\varphi_{ih}$ of $n-1$ arguments,
to which the unknown function $u$ should reduce when restricted to the 
hyperplane $\{x_i=\bar x_i\}$ through $\bar x$.
We use the following notation: 
for any $x\in \RR^n$ and $\xi\in\RR$ set
\beq\label{notn_1}
	x\,'^{\, i}:=(x_1,\dots,x_{i-1},x_{i+1},\dots,x_n)\in\RR^{n-1}
\eeq
and
\beq\label{notn_2}
	x\,'^{\, i}_\xi:=(x_1,\dots,x_{i-1},\xi,x_{i+1},\dots,x_n)\in\RR^{n}.
\eeq
The data requirement is then that each component $u_{ih}$ satisfy 
\beq\label{Ddata}
	u_{ih}(x\,'^{\, i}_{\bar x_i})
	=\varphi_{ih}(x\,'^{\, i}).
\eeq
We define the point $\bar\vp\in\RR^N$ by giving its components according 
to \eq{u_compns} as
\beq\label{phi_bar}
	\bar\vp_{ih}:=\vp_{ih}(\bar x\,'^{\, i}).
\eeq
Of course, it is assumed that each function $\vp_{ih}$ is defined on a 
neighborhood of $\bar x\,'^{\, i}$, and that each function $f_{ih}$ is defined 
on a neighborhood of $(\bar x,\bar \vp)$.
For convenience we shall equip any Euclidean space $\RR^k$ under consideration 
with the $1$-norm
\[|y|:=\sum_{j=1}^k|y_i|,\qquad y\in \RR^k,\]
and let $B_r(y)$ denote the corresponding closed ball
\[B_r(y):=\{z\in\RR^k\,:\,|z|\leq r\}.\]

\begin{remark}
	A special case is that all dependent variables appear 
	differentiated with respect to the same independent variable, $x_1$ say.
	In this case the system reduces to a system of ODEs with parameters $x_2,\dots,x_n$,
	and the data consist of $N$ functions of $x_2,\dots,x_n$ which the solution 
	should reduce to when $x_1=\bar x_1$. If in addition $x_1$ is the only 
	independent variable then we would prescribe $N$ constants as data at 
	$x_1=\bar x_1$ (pure ODE case).
\end{remark}

Before formulating and proving Darboux's first theorem we 
include a representative example.

\begin{example}\label{example5}
	Consider the system
	\begin{align}
		u_x &= f(x,y,z,u,v,w,\xi)\label{51}\\
		v_x &= g(x,y,z,u,v,w,\xi)\label{52}\\
		w_y &= h(x,y,z,u,v,w,\xi)\label{53}\\
		\xi_y &= k(x,y,z,u,v,w,\xi)\,,\label{54}
	\end{align}
	with independent variables $x_1=x$, $x_2=y$, $x_3=z$, and dependent variables
	$u_{11}=u$, $u_{12}=v$, $u_{21}=w$, $u_{22}=\xi$. (Here $n=3$, $N=4$, $N_1=2$, $N_2=2$,
	$N_3=0$,
	and $f_{11}=f$, $f_{12}=g$, $f_{21}=h$, $f_{22}=k$.)  The type of data considered in Darboux's first
	theorem are given as follows: near a given point $(\bar x,\bar y)\in\RR^2$ we prescribe four
	functions $\varphi_{11}(y,z)$, $\varphi_{12}(y,z)$, $\varphi_{21}(x,z)$, and $\varphi_{22}(x,z)$.
	Darboux's first theorem then guarantees the existence and uniqueness of a solution 
	$(u(x,y),v(x,y),w(x,y))$ to the system \eq{51}-\eq{54} near $(\bar x,\bar y)$ 
	which satisfies 
	\begin{align}
		u(\bar x,y,z) &= \varphi_{11}(y,z)\\
		v(\bar x,y,z) &= \varphi_{12}(y,z)\\
		w(x, \bar y,z) &= \varphi_{21}(x,z)\\
		\xi(x, \bar y,z) &= \varphi_{22}(x,z)\,.
	\end{align}
	In this case the $z$-variable plays the role of a parameter.
\end{example}

\subsection{Statement and proof of Darboux's first theorem} 
Recall that by a ``solution'' we mean a classical, $C^1$ solution satisfying 
the PDEs in \eq{sys_i}, as well as the data requirements \eq{Ddata}, in a 
pointwise manner.

\begin{theorem}\label{darb_1}
With the notations introduced above, assume there 
exist numbers $a,\, b,\, L>0$ such that the functions $f_{ih}$ and 
$\vp_{ih}$ ($1\leq i\leq n$, $1\leq h\leq N_i$)  satisfy:
\begin{itemize}
	\item[(A1)] Each $f_{ih}$ maps 
	$B_{a,b}:=B_a(\bar x)\times B_b(\bar \vp)$
	continuously into $\RR$, with
	\beq\label{f_lip}
		\qquad |f_{ih}(x,u)-f_{ih}(x,v)|\leq L|u-v| \qquad \text{for $(x,u),(x,v)\in B_{a,b}$.}
	\eeq
	\item[(A2)] Each $\vp_{ih}$ maps $B_a(\bar x\,'^{\, i})$ continuously 
	into $\RR$, with
	\beq\label{phi_cont}
		|\vp_{ih}(x\,'^{\, i})-\bar\vp_{ih}|\leq \textstyle\frac{b}{2N} \qquad \text{for $x\,'^{\, i}\in B_a(\bar x\,'^{\, i})$}.
	\eeq
	\end{itemize}
	Then, with 
	\[M:=\max_{i,h}\sup\{|f_{ih}(x,u)|\,:\, (x,u)\in B_{a,b}\} \quad\text{and}\quad
	\sigma:=\min\big(a,\textstyle\frac{b}{2MN}\big),\]
	the PDE system
	\beq\label{Syst1}
		\frac{\del u_{ih}}{\del x_i} = f_{ih}(x_1,\dots,x_n,u_1,\dots,u_N),
	\eeq
	with data
	\beq\label{Data1}
		u_{ih}(x\,'^{\, i}_{\bar x_i})=\varphi_{ih}(x\,'^{\, i}),
	\eeq
	has a unique solution $u(x)$ defined on $B_\sigma(\bar x)$. 
\end{theorem}
\begin{remark}
	Given $f_{ih}$ satisfying (A1) and continuous $\vp_{ih}$, we can 
	satisfy \eq{phi_cont} by reducing $a$ if necessary.
\end{remark} 
\begin{proof}
	We follow the standard procedure of realizing the solution as 
	the fixed point of a map defined on a suitable space of continuous 
	functions. For this, define the function $\vp:B_a(\bar x)\to\RR^N$ 
	by giving its components according to \eq{u_compns} as 
	\[\vp(x)_{ih}:=\vp_{ih}(x\,'^{\, i}),\]
	and then set
	\beq\label{X}
		X:=\left\{u\in C\!\left(B_\si(\bar x);\RR^N\right)\,:\, |u(x)-\vp(x)|\leq \textstyle\frac{b}{2}
		\quad\forall x\in B_\si(\bar x)\right\}.
	\eeq
	Equipped with $\sup$-metric the set $X$ is complete; for convenience we equip 
	$X$ with the following equivalent metric $d$: 
	\beq\label{d}
		d(u,v):=\sup_{x\in B_\si(\bar x)} e^{-K|x-\bar x|}|u(x)-v(x)|,\qquad\text{for $u, v\in X$},
	\eeq
	where 
	\beq\label{K}
		K:=2LN.
	\eeq
	On the complete metric space $(X,d)$ we define the functional map
	$u\mapsto \Phi[u]$ by giving its components according to \eq{u_compns}: 
	for $u\in X$ and $x\in B_\si(\bar x)$, set
        \beq\label{Phi}
            	\Phi[u]_{ih}(x):=\vp_{ih}(x\,'^{\, i})
            	+\int_{\bar x_i}^{x_i} f_{ih}(x\,'^{\, i}_\xi,u(x\,'^{\, i}_\xi))\, d\xi.
        \eeq
The first thing to verify is that $\Phi[u](x)$ is in fact well-defined whenever 
$u\in X$ and $x\in B_\si(\bar x)$. The first term on RHS\eq{Phi}
is defined since $\si\leq a$. For the integral on RHS\eq{Phi} we 
note that whenever $\xi$ is between $\bar x_i$ and $x_i$, then 
\beq\label{intermed}
	|x\,'^{\, i}_\xi-\bar x|\leq |x-\bar x|\leq \si\leq a,
\eeq
and 
\begin{align*}
	|u(x\,'^{\, i}_\xi)-\bar\vp|&\leq |u(x\,'^{\, i}_\xi)-\vp(x\,'^{\, i}_\xi)|+|\vp(x\,'^{\, i}_\xi)-\bar\vp|\\
	&\leq \textstyle\frac{b}{2}+\sum_{i,h}|\vp_{ih}(x\,'^{\, i})-\bar \vp_{ih}|\qquad\text{(since $u\in X$)}\\
	&\leq \textstyle\frac{b}{2}+N\cdot\textstyle\frac{b}{2N}=b.\qquad\qquad\qquad\text{(by (A2))}\\
\end{align*}
Thus, whenever $\xi$ is between $\bar x_i$ and $x_i$, and $u\in X$, then 
\[(x\,'^{\, i}_\xi,u(x\,'^{\, i}_\xi))\in B_{a,b},\]
such that $f_{ih}(x\,'^{\, i}_\xi,u(x\,'^{\, i}_\xi))$ is defined according to (A1). This shows that 
each $\Phi[u]_{ih}(x)$ is defined whenever $x\in B_\si(\bar x)$ and $u\in X$.

Next, we show that $\Phi$ maps $X$ into itself. For $x$ and $u$ as above, clearly
$\Phi[u]$ is a continuous map. Furthermore,
\begin{align*}
	|\Phi[u](x)-\vp(x)|&= \sum_{i,h}|\Phi[u]_{ih}(x)-\vp_{ih}( x\,'^{\, i})|
	= \sum_{i,h}\big|\int_{\bar x_i}^{x_i} f_{ih}(x\,'^{\, i}_\xi,u(x\,'^{\, i}_\xi))\, d\xi\big|\\
	&\leq M\cdot \sum_{i,h}|x_i-\bar x_i|\leq MN|x-\bar x|\leq MN\si \leq \textstyle\frac{b}{2},
\end{align*}
which shows that $\Phi[u]$ belongs to $X$ whenever $u\in X$.

Finally, we argue that $\Phi:X\to X$ is a strict contraction. For this assume 
$u$, $v\in X$ and $x\in B_\si(\bar x)$, and compute:
\begin{align*}
	|\Phi[u](x)-\Phi[v](x)|&= \sum_{i,h}|\Phi[u]_{ih}(x)-\Phi[v]_{ih}(x)|\\
	&= \sum_{i=1}^n\sum_{h=1}^{N_i}\big|\int_{\bar x_i}^{x_i} f_{ih}(x\,'^{\, i}_\xi,u(x\,'^{\, i}_\xi))
	-f_{ih}(x\,'^{\, i}_\xi,v(x\,'^{\, i}_\xi))\, d\xi\big|\\
	&\leq \sum_{i=1}^n N_iL\cdot \int_{\bar x_i\wedge x_i}^{\bar x_i\vee x_i}|u(x\,'^{\, i}_\xi)-v(x\,'^{\, i}_\xi)|\, d\xi\\
	&\leq L\cdot d(u,v) \cdot \sum_{i=1}^n N_i\int_{\bar x_i\wedge x_i}^{\bar x_i\vee x_i}e^{K|x\,'^{\, i}_\xi-\bar x|}\, d\xi.
\end{align*}
(Here $\wedge$ and $\vee$ denote ``minimum'' and ``maximum,'' respectively.)
From this and the choice \eq{K} for $K$ we obtain that (recall $|\cdot|$ denotes  
1-norm and that $N$ is given by \eq{N})
\begin{align*}
	e^{-K|x-\bar x|}|\Phi[u](x)-\Phi[v](x)|&\leq 
	L\cdot d(u,v)\cdot  \sum_{i=1}^n N_ie^{-K|x_i-\bar x_i|}\int_{\bar x_i\wedge x_i}^{\bar x_i\vee x_i}e^{K|\xi-\bar x_i|}\, d\xi\\
	&= L\cdot d(u,v)\cdot  \sum_{i=1}^n N_ie^{-K|x_i-\bar x_i|}\frac{1}{K}\big(e^{K|x_i-\bar x_i|}-1\big)\\
	&< L\cdot d(u,v)\cdot \frac{N}{K}=\textstyle\frac{1}{2}d(u,v).
\end{align*}
It follows from \eq{d} that 
\[d(\Phi[u],\Phi[v])\leq\textstyle\frac{1}{2}d(u,v),\]
i.e.\ $\Phi:X\to X$ is a strict contraction. 
According to the contraction mapping theorem
$\Phi$ has a unique fixed point $u\in X$, i.e.
\[u_{ih}(x)=\vp_{ih}(x\,'^{\, i})+\int_{\bar x_i}^{x_i} f_{ih}(x\,'^{\, i}_\xi,u(x\,'^{\, i}_\xi))\, d\xi\]
for each $1\leq i\leq n$, $1\leq h\leq N_i$. It is immediate that the PDEs in \eq{Syst1} are
satisfied, and that $u$ satisfies the data requirements in \eq{Data1}.

Finally, uniqueness follows from the fact that any solution (in the classical, pointwise sense 
that we consider) is a fixed point of $\Phi$, together with the uniqueness of such fixed points.
\end{proof}

\section{Darboux's theorem on overdetermined systems}\label{darb_3_prelims}
\subsection{Preliminaries}\label{prelims}
By ``Darboux's theorem on overdetermined systems'' we mean 
Th\'eor\`eme III of Chapter I of Livre III in \cite{dar}. 
From here on we refer to this simply as ``Darboux's theorem,'' while 
Theorem \ref{darb_1} above will be called ``Darboux's first theorem.''

Darboux's theorem  requires more work to state and prove. In this 
section we first describe the class of systems 
under consideration, and then give two examples. Finally, we comment 
on Darboux's original treatment which provided a proof for the cases of 
two and three independent variables. Then, in Section \ref{darb_3}, we 
introduce the class of ``Darboux systems'' and formulate Darboux's 
theorem for these. The proof of the existence part of the theorem proceeds
by induction on the number of independent variables. To highlight its structure 
we outline the argument for the case $n=3$ in Section \ref{darb_n=3}. 
The detailed proof  for  an arbitrary number of independent variables is carried out 
in Section \ref{gen_case_proof}.

\begin{convention}\label{conv_1}
       	In what follows ``an unknown'' refers to a dependent variable 
	that is to be solved for. We allow for an unknown to be a vector
	of unknown, scalar functions.
	The independent variables are denoted by $x=(x_1,\dots,x_n)\in \RR^n$. 
\end{convention}
We now change notation slightly from earlier and let $u$ denote an unknown, 
while $U$ will denote a list of all unknowns. The system we consider 
prescribes certain (possibly more than one) partials $u_{x_i}$ in terms 
of $x$ and some of the other unknowns (possibly all or none of them). 
In particular, the systems we consider are assumed to be closed in the sense that 
each unknown appears differentiated with respect to at least one $x_i$.
Thus, each equation has the general form
\beq\label{gen_form_1}
	u_{x_i}(x)=F^{u}_i(x;U(x)),
\eeq
where $F^u_i$ is a given function. 
Whenever \eq{gen_form_1} appears in the given system, we say that the system
{\em prescribes} $u_{x_i}$.
At this stage, before we consider integrability 
conditions, each $F^{u}_i$ may depend on any collection of unknowns. 
Without loss of generality we make the following:
\begin{convention}\label{conv_2}
        In writing down the equations \eq{gen_form_1}, it is assumed that all 
        scalar unknowns for which the system \eq{gen_form_1} prescribes 
        exactly the same partials, are already grouped together in a single 
        vector $u$ of unknowns. 
\end{convention}
According to this convention, for a given system of the form \eq{gen_form_1}, 
there is a one-to-one correspondence between the set of unknowns $U$ and 
a certain set $\I$ of strictly increasing multi-indices over $\{1,\dots,n\}$. Namely,
for each unknown $u$ in \eq{gen_form_1} we can associate a unique, strictly 
increasing multi-index $I_u:=(i_1,\dots,i_m)$ ($1\leq i_1<\cdots<i_m\leq n$) 
with the property that the given system \eq{gen_form_1} prescribes exactly 
the partials $u_{x_{i_1}},\dots,u_{x_{i_m}}$. If $I_u$ only contains one index $i$,
we write $I_u=(i)$.

The set of multi-indices $\I$ is available once the system \eq{gen_form_1} is 
given, and it is convenient to use $\I$ to index the unknowns.
Thus, for each $I\in\I$ we let $u^I$ denote the (according to Convention \ref{conv_2}, 
unique) unknown for which \eq{gen_form_1} prescribes exactly the partials 
$u_{x_i}$, $i\in I$.\footnote{Here and below we use the shorthand notation $i\in I$
to mean that $i$ is one of the indices in $I$.} 
With this notation, and upon renaming the right-hand sides in 
\eq{gen_form_1}, we obtain a system of the following form: 
for each $I\in \I$ the unknown $u^I$ satisfies the equations
\beq\label{gen_form_2}
	u^I_{x_i}(x)=F^I_i(x;U(x)) \qquad \text{for each $i\in I$.}
\eeq

To describe the data for the system \eq{gen_form_2}, let $\bar x\in \RR^n$ 
be a given point and assume $u^I$ is an unknown. With $I=(i_1,\dots,i_m)$, 
we let $I^c$ denote the strictly increasing multi-index of indices $i$ not 
belonging to $I$, we let
\[x_I:=(x_{i_1},\dots,x_{i_m}),\]
and similarly define $x_{I^c}$.
We then give data that prescribe the unknown
$u^I$ along the hyperplane through $\bar x$ spanned by those direction for which the 
system does not prescribe its partials. That is, we require 
\beq\label{orig_gen_data}
	u^I(x)|_{x_I=\bar x_I}=\bar u^I(x_{I^c}),
\eeq
where $\bar u^I$ is a given function.

For convenience we introduce the following terminology.
\begin{definition}\label{ov_cl_max}
	The system \eq{gen_form_2} is said to be {\em overdetermined} 
	provided that there is at least one unknown for which more than 
	one partial derivative is prescribed by the system.
\end{definition}
From now on it is assumed that the given system \eq{gen_form_2} 
under consideration is overdetermined. (If it is not then we are in a 
situation covered by Darboux's 1st theorem.) 
For an overdetermined system we need to impose integrability conditions, 
and these will put constraints on which unknowns the functions $F^I_i$ 
on the right hand side of \eq{gen_form_2} can 
depend on. This will be detailed in Section \ref{darb_3} below. We first
consider the simplest situations where Darboux's theorem applies.

\subsection{Two examples}\label{examples}
We consider two examples with $n=2$ independent variables.
For brevity we do not employ the index notation introduced above.
\begin{example}\label{3rd_ex1}
	The simplest situation where Darboux's theorem applies 
	is the following: 3 equations for 2 (scalar or vector) unknowns in 2 independent 
	variables. Let the unknowns be $u$ and $w$, let
	the independent variables be $x$ and $y$, and assume that the 
	equations are
	\begin{align}
		u_x &= F(x,y,u,w)\label{3rd_ex1_1}\\
		w_x &= f(x,y,w)\label{3rd_ex1_2}\\
		w_y &= \vp(x,y,u,w)\,.\label{3rd_ex1_3}
	\end{align}
	The data take the form
	\begin{align}
		u(\bar x,y) &= \bar u(y)\label{3rd_ex1_4}\\
		w(\bar x,\bar y) &= \bar w,\label{3rd_ex1_5}
	\end{align}
	where $\bar u$ is a given function and $\bar w$ is a given constant.
	Note that $f$ does not depend explicitly on $u$. This is a consequence of the
	single integrability condition in this case: ``$(w_x)_y=(w_y)_x$'' 
	or $\del_y [f(x,y,u,w)]=\del_x [\vp(x,y,u,w)]$.  The 
	requirement that this last equation, after applying the chain rule, should involve 
	only partials of $u$ or $w$ that are prescribed by \eq{3rd_ex1_1}-\eq{3rd_ex1_3}, 
	implies that $f$ must be independent of $u$. Indeed, the only 
	$u_y$-term is $f_uu_y$, and as $u_y$ is not prescribed by the system, we 
	need to assume that $f_u\equiv 0$.
	Substituting from \eq{3rd_ex1_2}-\eq{3rd_ex1_3}, we obtain the requirement that  
	\[f_y+f_w\vp=\vp_x+\vp_uF+\vp_wf.\]
	(If $w$, say, is a vector of unknowns, then $f_w$ denotes the Jacobian
	matrix of $f$ with respect to $w$.)
	This integrability condition need to be imposed as an identity in $(x,y,u,w)$, i.e.\
	we require that
	\begin{align}
		&f_y(x,y,w)+f_w(x,y,w)\vp(x,y,u,w)\nn\\
		&\quad = \vp_x(x,y,u,w)+\vp_u(x,y,u,w)F(x,y,u,w)
		+\vp_w(x,y,u,w)f(x,y,w),\nn
	\end{align}
	for all $(x,y,u,w)\approx (\bar x,\bar y,\bar u(\bar y), \bar w)$.
	Under this condition, and some mild regularity assumptions, Darboux's theorem 
	will guarantee the existence of a unique local solution $(u(x,y),w(x,y))$ 
	to \eq{3rd_ex1_1}-\eq{3rd_ex1_3} near $(\bar x,\bar y)$ that takes on the data 
	\eq{3rd_ex1_4}-\eq{3rd_ex1_5}.
\end{example}
Note that the system \eq{3rd_ex1_1}-\eq{3rd_ex1_3} is not ``maximal'' in 
the sense that, while $y$ is an independent variable, there is no unknown 
that appears differentiated with respect to only $y$. The next example 
considers the simplest case of a maximal system.
\begin{example}\label{3rd_ex2}
	Consider a system with 4 equations for 3 (scalar or vector) unknowns in 
	2 independent variables: let the unknowns be $u$, $v$, $w$, let the 
	independent variables be $x$ and $y$, and assume that the equations are
	\begin{align}
	u_x &= F(x,y,u,v,w)\label{3rd_case1_1}\\
	v_y &= \Phi(x,y,u,v,w)\label{3rd_case1_2}\\
	w_x &= f(x,y,v,w)\label{3rd_case1_3}\\
	w_y &= \vp(x,y,u,w)\,.\label{3rd_case1_4}
	\end{align}
	The data are prescribed near a point $(\bar x,\bar y)$ and take the following form:
	\begin{align}
	        	u(\bar x,y) &= \bar u(y)\label{3rd_case1_1_data}\\
        		v(x,\bar y) &= \bar v(x)\label{3rd_case1_2_data}\\
        		w(\bar x,\bar y) &= \bar w\label{3rd_case1_3_data},
	\end{align}
        where $\bar u(y)$ is defined near $\bar y$, $\bar v(x)$ is defined 
        near $\bar x$, and $\bar w$ is a  constant. 
        
	As in Example \ref{3rd_ex1}, the requirement that the  condition 
        ``$(w_x)_y=(w_y)_x$'' should not involve the unspecified partials 
        $u_y$ and $v_x$, implies that $f$ is independent of $u$ and $\vp$
        is independent of $v$. Substituting from the system we obtain that 
        \beq\label{int_condn}
	        	f_y+f_v\Phi+f_w\vp=\vp_x+\vp_uF+\vp_wf
        \eeq
	should hold as an identity for all $(x,y,u,v,w)\approx 
	(\bar x,\bar y,\varphi_{1}(\bar y),\varphi_{2}(\bar x),\bar \varphi_{3})$.
        Then, under suitable regularity assumptions, Darboux's theorem 
        guarantees the existence of a unique, local solution 
        $(u(x,y),v(x,y),w(x,y))$ to \eq{3rd_case1_1}-\eq{3rd_case1_4} 
        near the point $(\bar x,\bar y)$ taking on the data \eq{3rd_case1_1_data}-
        \eq{3rd_case1_3_data}.
\end{example}

\subsection{Comments on Darboux's formulation and proofs}\label{comments}
While Darboux \cite{dar} stated his theorem (Theorem \ref{d_thm} below) 
for any number $n$ of independent 
variables, he provided proofs only for the cases $n=2$ and $n=3$. 
In fact, by letting the unknowns $u$, $v$, $w$ in Example \ref{3rd_ex2} denote
vectors, we obtain precisely the setting that Darboux considered for $n=2$. 
(Darboux worked at the level of individual, scalar unknowns, which required an
extra layer of indices.)

Concerning Darboux's proofs, we note that his proof for the $n=2$ case 
makes use of his first theorem, i.e.\ Theorem \ref{darb_1} above. 
His proof for the $n=3$ case then makes use of both the result for $n=2$, 
as well as his first theorem. This indicates
that one can obtain a proof for any number of independent variables via
an induction argument. This is what we do below. In particular, we have 
not been able to provide a ``direct'' proof along the same lines as in the 
proof of his first theorem (Theorem \ref{darb_1} above). 

Let's elaborate. For an overdetermined system satisfying the 
assumptions of Darboux's theorem, it is easy enough 
to define a suitable functional 
map (corresponding to $\Phi$ in the proof of Theorem 
\ref{darb_1}), with the property that the sought-for solution is a fixed point of the 
functional map. However, we have not been able to give a proof for 
the required, opposite implication: if $U$ is fixed point of the 
functional map, then $U$ is the sought-for solution.

Instead, for the existence part, we shall extend Darboux's strategy for 
$n=2$ and $n=3$: induction on the number of independent variables, 
and repeated use of his first theorem. Finally, uniqueness will follow from the 
inductive construction we employ.

\section{Formulation of Darboux's theorem}\label{darb_3}
We now return to the general system \eq{gen_form_2}, which is 
assumed to be overdetermined. Below we describe the 
relevant integrability conditions and introduce the resulting class 
of ``Darboux systems.'' We then state Darboux's theorem which 
amounts to the unique, local solvability of such systems under 
certain mild regularity conditions.
\subsection{Integrability conditions}\label{gen_int_conds}
Let the system \eq{gen_form_2} be given and let $\I$ denote the 
corresponding set of multi-indices $I$ as detailed in Section \ref{prelims}.
If $u^I(x)$ is part of a solution $U(x)$ to \eq{gen_form_2}, and 
$i,j\in I$ with $i\neq j$, then we must have that  
$(u^I_{x_i})_{x_j}=(u^I_{x_j})_{x_i}$. 
With the notation above this amounts to having
\begin{align}
	&F^I_{i,x_j}(x;U(x))+\sum_{J\in\I} F^I_{i,u^J}(x;U(x))u^J_{x_j}(x)\nn\\
	&=F^I_{j,x_i}(x;U(x))+\sum_{J\in\I} F^I_{j,u^J}(x;U(x))u^J_{x_i}(x),
	\label{int_cond1}
\end{align}
Here $F^I_{i,x_j}$ denotes the partial derivative of $F^I_i$ with respect to $x_j$,
while $F^I_{i,u^J}$ denotes the Jacobian matrix of $F^I_i$ with respect to $u^J$.
For \eq{int_cond1} to hold we must require that:
\begin{itemize}
	\item[(i)] all partials of unknowns $u^J$ appearing in \eq{int_cond1} 
	are also prescribed by the original system \eq{gen_form_2}, and that
	\item[(ii)] upon substitution from the original system for these, 
	an identity in $(x,U)$ is obtained.
\end{itemize}
As in the examples above, the condition in (i) places constraints on which 
unknowns each $F^I_i$ can depend on. Namely, we need to require that 
$F^I_i$ depends explicitly on (at most) those unknowns $u^J$ 
with the property that each index in $I\minus i$ also belongs to $J$.  
(Here we employ the following shorthand notation:  if $I=(i_1,\dots,i_m)$ 
and $i=i_j$ for some $1\leq j\leq m$, then $I\minus i$ denotes the multi-index 
$(i_1,\dots i_{j-1},i_{j+1},\dots,i_m)$.) We express this by writing $J\supset I\minus i$.
Next, whenever $i\in I\in \I$ we define 
\beq\label{I^I_i}
	\I^I_i:=\{J\in\I\,|\, J\supset I\minus i\}
\eeq
together with the corresponding set of unknowns
\beq\label{U^I_i}
	U^I_i:=\{u^J\,|\, J\in\I^I_i\}.
\eeq
With this notation, condition (i) requires that each $F^I_i$ appearing in the system
depends explicitly on (at most) the unknowns in $U^I_i$.
Thus (i) amounts to requiring that the equations for the unknown $u^I$ take the form
\beq\label{gen_form_3}
	u^I_{x_i}(x)=F^I_i(x;U^I_i(x))\qquad \text{whenever $i\in I\in \I$,}
\eeq
Condition (ii) then becomes the requirement that: whenever $I\in\I$, $i,j\in I$, and 
$i\neq j$, then
\begin{align}
	&F^I_{i,x_j}(x;U^I_i)+\sum_{J\in\I^I_i} F^I_{i,u^J}(x;U^I_i)F^J_j(x;U^J_j)\nn\\
	&= F^I_{j,x_i}(x;U^I_j)+\sum_{J\in\I^I_j} F^I_{j,u^J}(x;U^I_j)F^J_i(x;U^J_i)
	\label{int_cond2}
\end{align}
holds as an identity in $(x,U)$ near $(\bar x,\bar U)$.
Note that, in the sum on the lefthand side of \eq{int_cond2}, each function $F^J_j$ is 
given by the original system \eq{gen_form_3}. This is because $J\in\I$, so that 
according our notation, $u^J$ is an unknown in the system, and $J\supset I\minus i$ 
together with $j\neq i$ imply that $j\in J$; thus $u^J_{x_j}$ is prescribed
by the system, i.e.\ $F^J_j$ appears as one of the righthand sides in \eq{gen_form_3}. 
The same remarks apply to the functions $F^J_i$ in the sum on righthand side of 
\eq{int_cond2}. 
\begin{definition}\label{d_type}
	Let $\I$ be a set of strictly increasing multi-indices over $\{1,\dots,n\}$,
	and let $u^I$, $I\in\I$, be the corresponding unknowns. 
	Then, with the notation introduced above, a system of equations 
	of the form \eq{gen_form_3}, where the maps $F^I_i(x;U^I_i)$ are 
	defined on a neighborhood of a given point $(\bar x,\bar U)$,
	is called a {\em Darboux system} near $(\bar x,\bar U)$ 
	provided the integrability conditions \eq{int_cond2} is satisfied.
\end{definition}
We can now formulate Darboux's theorem as follows.
\begin{theorem}[Darboux's theorem]\label{d_thm}
	Assume that \eq{gen_form_3} is a Darboux system near $(\bar x,\bar U)$
	according to Definition \ref{d_type}, and let the data \eq{orig_gen_data} 
	be assigned on a neighborhood of $(\bar x,\bar U)$ for each unknown $u^I$. 
	Assume that all functions $F^I_i$ in \eq{gen_form_3} are $C^2$-smooth near 
	$(\bar x,\bar U)$, and that all functions $\bar u^I$ in \eq{orig_gen_data} 
	are $C^2$-smooth near $\bar x$. Then there is a neighborhood 
	of $\bar x$ on which the system \eq{gen_form_3} has a unique $C^2$-smooth
	solution taking on the assigned data \eq{orig_gen_data}.
\end{theorem}
\begin{remark}\label{frob_rmk}
	Note that we allow for the possibility that the integrability conditions \eq{int_cond2}
	are vacuously met. However, in that case each unknown appears 	
	differentiated with respect to only one independent variable. Since we assume that the given 
	system is closed, this case is covered by Darboux's first theorem (Theorem \ref{darb_1} above).
	In what follows it therefore suffices to restrict attention to overdetermined Darboux 
	systems.
	
	We also remark that Darboux's theorem also covers the opposite extereme
	case where all first partials of all unknowns are prescribed, and the corresponding
	integrability conditions are met. This result is the standard Frobenius theorem for
	overdetermined systems, cf.\ Theorem 1, Chapter 6 in \cite{spi}.
\end{remark}

\section{Outline of proof for the case $n=3$}
\label{darb_n=3} 
To motivate the structure of the proof for general $n$, 
we provide an outline of Darboux's proof for the case $n=3$. 
Although our general proof, when specialized to $n=3$, will differ slightly from 
Darboux's proof, it will help to explain the structure of the proof for arbitrary $n$.
We recall that \cite{dar} establishes the existence part of Darboux's 
theorem for maximal systems when the number of independent 
variables is $n=2$. This provides the base-step for the inductive proof.

We consider a Darboux system in the three independent variables. For brevity
we do not apply the index notation introduced above and instead follow (partially)
the notation of Darboux \cite{dar}. The independent variables are
$X:=(x,y,z)$ and the independent variables are the seven (scalar, say) 
unknowns $U:=(u,v,w,p,q,r,s)$.
For concreteness we assume that the PDE system under consideration 
is maximal in the sense that for every choice of one, two, or three 
independent variables, there is an unknown which appears differentiated
with respect to exactly the chosen variables. Without loss of generality 
we assume the data are prescribed on hyperplanes through $\bar X=0$.
The system thus consists of 12 equations of the following form:
\begin{align}
	u_x &= F(X;U)\label{pde1}\\
	v_y &= \Phi(X;U)\label{pde2}\\
	w_z &= \Psi(X;U)\label{pde3}\\
	p_y &= f_1(X;U)\label{pde4}\\
	p_z &= f_2(X;U)\label{pde5}\\
	q_x &= \vp_0(X;U)\label{pde6}\\
	q_z &= \vp_2(X;U)\label{pde7}\\
	r_x &= \psi_0(X;U)\label{pde8}\\
	r_y &= \psi_1(X;U)\label{pde9}\\
	s_x &= \theta_0(X;U)\label{pde10}\\
	s_y &= \theta_1(X;U)\label{pde11}\\
	s_z &= \theta_2(X;U)\label{pde12},
\end{align}
with prescribed data of the form
\begin{align}
	u(0,y,z) &= \bar u(y,z)\label{data1}\\
	v(x,0,z) &= \bar v(x,z)\label{data2}\\
	w(x,y,0) &= \bar w(x,y)\label{data3}\\
	p(x,0,0) &= \bar p(x)\label{data4}\\
	q(0,y,0) &= \bar q(y)\label{data5}\\
	r(0,0,z) &= \bar r(z)\label{data6}\\
	s(0,0,0) &=\bar s.\label{data7}
\end{align}
The system is clearly overdetermined; to be a Darboux system the
two constraints (i) and (ii) above must be satisfied. It is straightforward to 
verify that the first constraint implies the following dependencies: 
\begin{align}
	&f_1 =f_1(X;w,p,q,s),\qquad f_2 =f_2(X;v,p,r,s),\\
	&\vp_0 =\vp_0(X;w,p,q,s),\qquad \vp_2 =\vp_2(X;u,q,r,s),\\
	&\psi_0 =\psi_0(X;v,p,r,s),\qquad \psi_1 =\psi_1(X;u,q,r,s),\\
	&\theta_0 =\theta_0(X;p,s),\qquad \theta_1 =\theta_1(X;q,s),\qquad
	\theta_2 =\theta_2(X;r,s).
\end{align}
The second constraint then requires that the following integrability 
conditions are satisfied as identities on a full 
$\RR^3_X\times\RR^7_U$-neighborhood of the point
$(0,\bar u(0),\bar v(0),\bar w(0),\bar p(0),\bar q(0),\bar r(0),\bar s)$:
\begin{align}
	&f_{1,z}+f_{1,w}\Psi+f_{1,p}f_2+f_{1,q}\vp_2+f_{1,s}\theta_2\nn\\
	&= f_{2,y}+f_{2,v}\Phi+f_{2,p}f_1+f_{2,r}\psi_1+f_{2,s}\theta_1,
	\label{reln_3}
\end{align}
\begin{align}
	&\vp_{0,z}+\vp_{0,w}\Psi+\vp_{0,p}f_2+\vp_{0,q}\vp_2+\vp_{0,s}\theta_2\nn\\
	&= \vp_{2,x}+\vp_{2,u}F+\vp_{2,q}\vp_0+\vp_{2,r}\psi_0+\vp_{2,s}\theta_0,
	\label{reln_6}
\end{align}
\begin{align}
	&\psi_{0,y}+\psi_{0,v}\Phi+\psi_{0,p}f_1+\psi_{0,r}\psi_1+\psi_{0,s}\theta_1\nn\\
	&= \psi_{1,x}+\psi_{1,u}F+\psi_{1,q}\vp_0+\psi_{1,r}\psi_0+\psi_{1,s}\theta_0,
	\label{reln_9}
\end{align}
\begin{align}
	&\theta_{0,y}+\theta_{0,p}f_1+\theta_{0,s}\theta_1
	= \theta_{1,x}+\theta_{1,q}\vp_0+\theta_{1,s}\theta_0\label{reln_13}\\
	&\theta_{0,z}+\theta_{0,p}f_2+\theta_{0,s}\theta_2
	= \theta_{2,x}+\theta_{2,r}\psi_0+\theta_{2,s}\theta_0\label{reln_14}\\
	&\theta_{1,z}+\theta_{1,q}\vp_2+\theta_{1,s}\theta_2
	= \theta_{2,y}+\theta_{2,r}\psi_1+\theta_{2,s}\theta_1.\label{reln_15}
\end{align}
%
The overall approach for establishing existence of a solution taking on 
the assigned data is to build the solution by solving several smaller, auxiliary 
systems. Each auxiliary system is either a Darboux system in two independent 
variables (and thus solvable according to the base-step $n=2$), or a determined 
system (solvable according to Darboux's first theorem). All equations 
occurring in these smaller systems will be copies of equations form the
original system \eq{pde1}-\eq{pde12}.

For the case $n=3$ there will be three auxiliary systems, which we refer to as 
system 1 through 3, respectively. System 3 will be the last system to be 
solved\footnote{This is not precisely what Darboux does in \cite{dar}; his argument 
appears to proceed in the opposite order from what we do.},
and this will be a determined system which is solved via an application of Darboux's 
first theorem. What we want to arrange is that the resulting solution of system 3 
is in fact the sought-for solution of the original system \eq{pde1}-\eq{pde12} with 
the original data \eq{data1}-\eq{data7}. The main issue is how to provide data
for system 3 so that this works out, and for this we make use of the solutions 
to systems 1 and 2. 

System 3 consists of copies of the $u_x$, $v_y$, $w_z$, $p_y$, $q_x$, $r_x$, 
and $s_x$ equations from the original, given system \eq{pde1}-\eq{pde12}. 
(In the general case, system $n$ will consist of copies of all equations 
$u^I_{x_i}=F^I_i(x;U^I_i)$ from \eq{gen_form_3} with $I\in\I$ and $i=\min I$.)
Note that these equations form a determined system. For clarity we denote the 
unknowns of system 3 by $\hat U=(\hat u,\hat v,\hat w,\hat p,\hat q,\hat r,\hat s)$. 
The challenge is to assign appropriate $\hat U$-data so that the resulting solution
solves \eq{pde1}-\eq{pde12} with data \eq{data1}-\eq{data7}. 
As system 3 is solved using Darboux's first theorem, the relevant $\hat U$-data must 
prescribe
\beq\label{remaining_data}
	\!\!\hat u(0,y,z),\, \hat v(x,0,z),\, \hat w(x,y,0),\, \hat p(x,0,z),\, \hat q(0,y,z),\, 
	\hat r(0,y,z),\, \hat s(0,y,z).
\eeq
For the three first we simply use the original data in \eq{data1}-\eq{data3}.
However, the remaining four pieces of data need to be generated, and for this
we proceed to identify and solve systems 1 and 2.

For system 1, we make copies of the $q_z$, $r_y$, $s_y$, and $s_z$ 
equations from the original system and restrict to the hyperplane $\{x=0\}$.
(In the general case, system 1 will consist of copies of all equations
$u^I_{x_i}=F^I_i(x;U^I_i)$ from \eq{gen_form_3} where $I\in\I$, $1\in I\neq (1)$ and $1<i\in I$,
which are then restricted to $\{x_1=0\}$.) For clarity we introduce new labels 
$\ul q$, $\ul r$, $\ul s$ for the resulting unknowns, which are functions of $(y,z)$.
System 1 thus takes the form
\begin{align}
	\ul q_z &= \vp_2(0,y,z;\bar u(y,z), (\ul q,\ul r,\ul s)(y,z))\label{1st_ss_1}\\
	\ul r_y &= \psi_1(0,y,z;\bar u(y,z), (\ul q,\ul r,\ul s)(y,z))\label{1st_ss_2}\\
	\ul s_y &= \theta_1(0,y,z;\bar u(y,z), (\ul q,\ul s)(y,z))\label{1st_ss_3}\\
	\ul s_z &= \theta_2(0,y,z;\bar u(y,z), (\ul r,\ul s)(y,z)).\label{1st_ss_4}
\end{align}
Observe that this is a closed system: each unknown 
appearing on one of the righthand sides in \eq{1st_ss_1}-\eq{1st_ss_4}
also appears at least once in one of the lefthand sides.
For this system we use the original data to assign the data
\[\ul q(y,0)=\bar q(y),\qquad \ul r(0,z)=\bar r(z),\qquad \ul s(0,0)=\bar s.\]
As the equations \eq{1st_ss_1}-\eq{1st_ss_4} are obtained by duplication
and restriction to $\{x=0\}$ of equations from the original system, it turns out
that system 1 is again a Darboux system, but now in the two independent 
variables $(y,z)$. According to the base-step of the proof it has a local solution 
$(\ul q,\ul r,\ul s)(y,z)$ defined for $(y,z)\approx (0,0)$ and taking on the assigned data.

At this point we can assign three more of the data in \eq{remaining_data} by setting 
\[\hat q(0,y,z)=\ul q(y,z),\quad \hat r(0,y,z)=\ul r(y,z),\quad\text{and}\quad
\hat s(0,y,z)=\ul s(y,z).\]
However, it still remains to assign appropriate data for $p(x,0,z)$. To do so 
we exploit the $p_z$ equation from the original system, which we have 
not yet used. Proceeding as for system 1 we make a copy of this equation,
and then restrict it to the hyperplane $\{y=0\}$. However, differently from what
occurred for system 1, we do not obtain a closed system in this way: we get a
single equation for $p(x,0,z)$ which contains both $r(x,0,z)$ and $s(x,0,z)$
on its righthand side. To obtain a closed system we add copies of the  
$r_x$ and $s_x$ equations (note that these also appear in system 3), 
and restrict them to $\{y=0\}$. We denote the resulting unknowns by $\check{p}$,
$\check{r}$, $\check{s}$; these are functions of $(x,z)$ and are required 
to satisfy the following system 2:
\begin{align}
	\check{p}_z &= f_2(x,0,z;\bar v(x,z), (\check{p},\check{r},\check{s})(x,z))\label{2nd_ss_1}\\
	\check{r}_x &= \psi_0(x,0,z;\bar v(x,z), (\check{p},\check{r},\check{s})(x,z))\label{2nd_ss_2}\\
	\check{s}_x &= \theta_0(x,0,z;(\check{p},\check{s})(x,z)).\label{2nd_ss_3}
\end{align}
(This system, which is also considered in Darboux's argument, is not exactly 
what our general approach below uses as system 2; see Remark \ref{Darbouxs_version}.)
The equations \eq{2nd_ss_1}-\eq{2nd_ss_3} form a determined system 
and to provide data we use both the original data as well as the solution to system 1:
\[\check{p}(x,0)=\bar p(x),\qquad \check{r} (0,z)=\bar r(z),\qquad \check{s} (0,z)=\ul s(0,z).\]
According to Darboux's first theorem it possesses a local solution 
$(\check{p},\check{r},\check{s})(x,z)$ near $(\bar x,\bar z)=(0,0)$.

We can now assign appropriate data for system 3:
\[\hat u(0,y,z)=\bar u(y,z),\quad \hat v(x,0,z)=\bar v(x,z),\quad \hat w(x,y,0)=\bar w(x,y),\]
\[\hat p(x,0,z)=\check p(x,z),\quad \hat q(0,y,z)=\ul q(y,z),\quad \hat r(0,y,z)=\ul r(y,z),\]
and
\[\hat s(0,y,z)=\ul s(y,z).\]
Note that this makes use of the original data, as well as the solutions of both 
system 1 and system 2.
Darboux's first theorem guarantees the existence of a solution $\hat U(x,y,z)$
of system 3 on a full neighborhood of $0\in\RR^3$.

The claim now is that this $\hat U$ is in fact a solution 
of the original system \eq{pde1}-\eq{pde12} and takes on the original data 
\eq{data1}-\eq{data7}. Thanks to the construction of the solutions to system 1 and 2, 
it is straightforward to verify that $\hat U$ attains the data \eq{data1}-\eq{data7}.
Next, as system 3 is a subsystem of the original system, all that remains to 
be verified is that $\hat U$ solves the remaining $p_z$, $q_z$, $r_y$, $s_y$, 
and $s_z$ equations in \eq{pde1}-\eq{pde12}. For this we 
follow Darboux and define the quantities
\begin{align}
	A(X)&:=\hat p_z(X)-f_2(X;(\hat v,\hat p,\hat r,\hat s)(X))\\
	B(X)&:=\hat q_z(X)-\vp_2(X;(\hat u,\hat q,\hat r,\hat s)(X))\\
	C(X)&:=\hat r_y(X)-\psi_1(X;(\hat u,\hat q,\hat r,\hat s)(X))\\
	D(X)&:=\hat s_y(X)-\theta_1(X;(\hat q,\hat s)(X))\\
	E(X)&:=\hat s_z(X)-\theta_2(X;(\hat r,\hat s)(X)).
\end{align}
Without going into the details (carried out for the general case below),
it turns out that these quantities solve a linear, homogeneous, and determined 
system with vanishing 
data near $0\in\RR^3$. (Here the integrability conditions
\eq{reln_3}-\eq{reln_15} are used.) 
The uniqueness part of Darboux's first theorem then 
implies that they must vanish identically near $0\in\RR^3$. 
This means precisely that $\hat U$ satisfies also the $p_z$, $q_z$, $r_y$, $s_y$, 
and $s_z$ equations in the original system, on a full neighborhood of $0\in\RR^3$.
This establishes the existence part of Darboux's theorem in the case $n=3$.

\begin{remark}\label{Darbouxs_version}
	The outline above essentially follows Darboux's original proof for the case
	$n=3$. As already noted, our argument for the general case with any 
	number $n$ of independent variables will differ slightly from that of Darboux 
	when specialized to $n=3$. Specifically, in our setup for $n=3$, system 2 
	will be an overdetermined Darboux system, rather than a determined system.
\end{remark}

\section{Proof of Darboux's theorem}\label{gen_case_proof}
We now return to the general case and consider a given Darboux system 
\eq{gen_form_3} in $n$ independent variables and with data \eq{orig_gen_data}.
To simplify the notation, and without loss of generality, we assume 
from now on that $\bar x=0$. Thus the assigned data are:
\beq\label{gen_data}
	u^I(x)|_{\{x_I=0\}}=\bar u^I(x_{I^c}),
\eeq
Our proof proceeds by induction on the number $n$ of independent variables.
As noted above, \cite{dar} provides a proof of the existence part of Darboux's 
theorem when the number of independent variables is $2$ and $3$. This takes
care of the two base-steps for the inductive proof of the general case. We now 
assume $n> 3$, and the induction hypothesis is that Theorem 
\ref{d_thm} holds for the cases with $n-2$ and $n-1$ independent variables.
Before providing a detailed proof, we provide an outline of the argument.

\subsection{Outline of proof for general $n$}
As in the case $n=3$, we begin by defining $n$ auxiliary systems which 
are referred to as system 1 through $n$. System $n$, whose solution 
will turn out to be the sought-for solution, consists of a copy of each 
equation $u^I_{x_i}=F^I_i(x;U^I_i)$ from \eq{gen_form_3} with $I\in\I$ and $i=\min I$.
This is a determined system, and the main part of the proof is about how to 
generate appropriate data for this system. As above this is accomplished
by solving the auxiliary systems 1 through $n-1$, 
and then using their solutions to provide data for system $n$. 

Once such data are available, Darboux's first theorem guarantees the
existence of a solution $\hat U$ of system $n$. The fact that $\hat U$ 
satisfies the original data \eq{gen_data} will be a direct consequence 
of the construction. On the other hand, to show that $\hat U$ also satisfies the 
remaining equations in the original system \eq{gen_form_3} (i.e.\ 
the equations not in system $n$), requires an argument. For this we 
follow Darbouxs approach in the case $n=3$ and show that the functions 
$\hat u^I_{x_i}-F^I_i(x;\hat U^I_i)$ where $I\in\I$, $|I|>1$, and $i\in I$ is 
such that $i\neq \min I$, solve a linear, homogeneous, and determined 
system with vanishing data. For this, the integrability 
conditions \eq{int_cond2} are utilized. The uniqueness part of 
Darboux's first theorem then implies that these quantities vanish 
identically near $\bar x=0$, completing the existence part of 
Darboux's theorem for the case of $n$ independent variables.

The remaining issue is to define and solve system 1 through $n-1$. 
For each $\ell=1,\dots,n-1$, we copy  
certain equations from the original system (detailes below), 
and restrict them to the hyperplane $\{x_\ell=0\}$. The selection is 
made so as to guarantee that the chosen set of equations 
form a Darboux system in $n-1$ independent variables. We refer to
the resulting system as ``system $\ell$.'' The data for system $\ell$ 
are provided by the original data \eq{gen_data}. (This is a 
technical point where our argument differs from that of Darboux \cite{dar},
cf.\ Remark \ref{Darbouxs_version}.) 
According to the induction hypothesis, it has a local solution near 
$0\in\RR^{n-1}$. 

Having solved system 1 through $(n-1)$, we use their solutions to 
provide the appropriate data for system $n$, and the argument for
the existence part of Darboux's theorem in $n$ independent variables 
is concluded as indicated above.

\subsection{System $\ell$ for $\ell=1,\dots,n-1$} 
To define these systems we introduce the following notation:
for $\ell=1,\dots,n-1$, we let
\beq\label{I_ell}
	\I_{\ell}:=\{I\in\I\,|\, \ell\in I\neq(\ell)\},
\eeq
and denote the unknowns of system $\ell$ by $u^{\ell,I}$ for $I\in \I_{\ell}$. 
These will be functions of $x\,'^{\, \ell}= (x_1,\dots,x_{\ell-1},x_{\ell+1},\dots,x_n)$. 
The equations of system $\ell$ are then obtained as follows: for each $I\in\I_\ell$
and for each $i\in I\minus\ell$ we make a copy of the equation occurring in 
equation \eq{gen_form_3}, restrict it to the hyperplane $\{x_\ell=0\}$, and 
rename any unknown $u^K$ appearing in it as $u^{\ell,K}$.

At this point we need to address a notational issue. 
Namely, assume $i\in I\in \I_\ell$ and $i\neq \ell$. 
According to the procedure above, system $\ell$ will then 
contain a copy of the equation 
\[u^I_{x_i}=F^I_i(x;U^I_i)\]
from the original system \eq{gen_form_3}; the corresponding equation 
in system $\ell$ is obtained by restricting this equation to the hyperplane 
$\{x_\ell=0\}$.
Now, in the particular case that $I$ is the 2-index $(\ell,i)$ or $(i,\ell)$, then 
$I\minus i=(\ell)$ and, according to \eq{I^I_i}-\eq{U^I_i},
$F^I_i$ may depend explicitly on the unknown $u^{(\ell)}$. In the corresponding equation 
in system $\ell$, any such occurrence of $u^{(\ell)}$ is to be replaced 
by the originally given data $\bar u^{(\ell)}(x\,'^{\, \ell})$ (because we are 
restricting to the hyperplane $\{x_\ell=0\}$). It is awkward to treat the 2-indices 
$(\ell,i)$ and $(i,\ell)$ separately, and we shall simply write the corresponding equations 
in system $\ell$ as
\beq\label{syst_ell}
	u^{\ell,I}_{x_i}(x\,'^{\, \ell})
	=F^I_i(x\,'^{\, \ell}_0;[\bar u^{(\ell)}(x\,'^{\, \ell})],U^{\ell,I}_i(x\,'^{\, \ell})),
\eeq
where we have used the notation in \eq{notn_2} for $x\,'^{\, \ell}_0$, and
the square brackets indicate that the bracketed term is present only when 
$I=(\ell,i)$ or $I=(i,\ell)$. Finally, $U^{\ell,I}_i$ denotes the collection 
of the remaining dependent-variable arguments of $F^I_i$, i.e.,
whenever $i\in I\in\I_\ell$, we set
\beq\label{I^ell^I_i}
	\I^{\ell,I}_i:=\{K\in\I\,|\,  K\supset I\minus i,\, K\neq (\ell)\}
\eeq 
and
\beq\label{U^ell^I_i}	
	U^{\ell,I}_i:=\{u^{\ell,K}\,|\, K\in\I^{\ell,I}_i\}.
\eeq
Thus, system $\ell$ consists of the following 
equations: for each $I\in \I_{\ell}$ we have the $|I|-1$ equations
in \eq{syst_ell} with $i\in I\minus \ell$.
As data for $u^{\ell,I}$ we make use of the original data \eq{gen_data} and require 
\beq\label{ell_data}
	u^{\ell,I}(x\,'^{\, \ell})\big|_{\{x_{I'^{\ell}}=0\}}=\bar u^I(x_{I^c}),
\eeq
where $I'^{\ell}$ denotes the multi-index $I\minus \ell$. We note that the 
right hand side in \eq{ell_data} is independent of $\ell$.

For later reference we record the following: whenever $\ell\in I$ and $|I|\geq 3$, then 
\beq\label{3_or_more}
	\left\{\begin{array}{l}
		\text{the square-bracketed term in \eq{syst_ell} is absent,}\\
		\I^{\ell,I}_i=\I^{I}_i\quad\text{by \eq{I^I_i}, and}\\
		u^{\ell,K}\in U^{\ell,I}_i\quad \text{if and only if $u^K\in U^I_i$.}
	\end{array} \right\}
\eeq
We now have that:
\begin{claim}\label{claim:syst_ell}
	For each $\ell=1,\dots,n-1$, system $\ell$ is a closed Darboux system 
	in $n-1$ independent variables.
\end{claim}
\begin{proof}
	To show that system $\ell$ is closed we need to argue that whenever 
	an unknown $u^{\ell,K}$ appears on the righthand side of one of the 
	equations in \eq{syst_ell}, then system $\ell$ also contains an equation 
	where $u^{\ell,K}$ appears on the lefthand side. 
	If $u^{\ell,K}$ appears on the righthand side of the equation \eq{syst_ell}, 
	then $K\supset I\minus i$, and $K\neq (\ell)$. As $\ell\in I$ and $\ell\neq i$, 
	we have that $\ell\in K\neq (\ell)$. It follows that $K\in\I_\ell$, and that for each 
	$k\in K\minus \ell$ the equation 
	\[u^{\ell,K}_{x_k}=F^K_k(x\,'^{\, \ell}_0;[\bar u^{(\ell)}(x\,'^{\, \ell})],U^{\ell,K}_k(x\,'^{\, \ell}))\] 
	is prescribed by system $\ell$.
	
	Next, we need to verify the integrability conditions corresponding to 
	equality of 2nd mixed partial derivatives for system $\ell$. I.e., we must argue that the 
	two properties corresponding to (i) and (ii) in Section \ref{gen_int_conds} hold in 
	the present situation. 
	For this, assume that system $\ell$ prescribes both $u^{\ell,I}_{x_i}$ and $u^{\ell,I}_{x_j}$,
	where $i\neq j$.
	In particular, this means that $i$, $j$, and $\ell$ all belong to $I$, and that they are 
	all distinct. It follows that $|I|\geq 3$, such that \eq{3_or_more} applies and the 
	square-bracketed argument on the righthand side of \eq{syst_ell} is absent. 
	Thus, the equations in question read:
	\beq\label{1eqni}
		u^{\ell,I}_{x_i}=F^I_i(x\,'^{\, \ell}_0;U^{\ell,I}_i(x\,'^{\, \ell})),
	\eeq
	and
	\beq\label{1eqnj}
		u^{\ell,I}_{x_j}=F^I_j(x\,'^{\, \ell}_0;U^{\ell,I}_j(x\,'^{\, \ell})),
	\eeq
	In particular, as both $i$ and $j$ are different from $\ell$, both $x_i$ and $x_j$ 
	appear as independent variables in \eq{1eqni} and \eq{1eqnj}. Applying $\del_{x_j}$ 
	to \eq{1eqni}, $\del_{x_i}$ to \eq{1eqnj}, and equating the results we obtain (dropping 
	arguments for now)
	\begin{align}
		F^I_{i,x_j}+\sum_{J\in \I^{\ell,I}_i}F^I_{i,u^J}u^{\ell,J}_{x_j}
		= F^I_{j,x_i}+\sum_{J\in \I^{\ell,I}_j}F^I_{j,u^J}u^{\ell,J}_{x_i}.\label{int_intermed}
	\end{align}
	The first issue is to argue that all the partials derivatives $u^{\ell,J}_{x_j}$ and 
	$u^{\ell,J}_{x_i}$ in the two sums in \eq{int_intermed} are all prescribed by system $\ell$. 
	Consider $u^{\ell,J}_{x_j}$ in the lefthand sum. This partial derivative is prescribed by 
	system $\ell$ provided $j\in I\minus \ell$ and $J\in \I_\ell$. The first condition is satisfied 
	since $j$ and $\ell$ are distinct and both belong to $I$. The second condition is met 
	because $J\in \I^{\ell,I}_i$, so that $J\supset I\minus i\ni j,\, \ell$; in particular, 
	$\ell\in J\neq (\ell)$, i.e.\ $J\in \I_\ell$. A similar argument shows that each partial 
	$u^{\ell,J}_{x_i}$ in the righthand sum of \eq{int_intermed} is prescribed by system $\ell$.
	
	Again, as $i$, $j$, $\ell$ are distinct and belong to $I$, we have that \eq{3_or_more} applies.
	Therefore, upon substituting from system $\ell$, we obtain from \eq{int_intermed}
	the requirement that
	\begin{align}
		&\quad F^I_{i,x_j}(x\,'^{\, \ell}_0;U^{\ell,I}_i)
		+\sum_{J\in \I^I_i}F^I_{i,u^{\ell,J}}(x\,'^{\, \ell}_0;U^{\ell,I}_i)
		F^J_j(x\,'^{\, \ell}_0;U^{\ell,J}_j)\nn\\
		&= F^I_{j,x_i}(x\,'^{\, \ell}_0;U^{\ell,I}_j)
		+\sum_{J\in \I^I_j}F^I_{j,u^{\ell,J}}(x\,'^{\, \ell}_0;U^{\ell,I}_j)
		F^J_i(x\,'^{\, \ell}_0;U^{\ell,J}_i)\label{syst1_int_conda}
	\end{align}
	should hold as an identity with respect to $x\,'^{\, \ell}$ and the dependent variables occurring.
	We finish the proof by arguing that this is a consequence of the integrability condition 
	\eq{int_cond2}. Indeed, since \eq{int_cond2} by assumption holds as an identity 
	in $(x,U)$, we may restrict \eq{int_cond2} 	to the hyperplane $\{x_\ell=0\}$.
	The result is that the following identity holds
	\begin{align}
		& F^I_{i,x_j}(x\,'^{\, \ell}_0; U^I_i)
		+\sum_{J\in\I^I_i} F^I_{i,u^J}(x\,'^{\, \ell}_0; U^I_i)
		F^J_j(x\,'^{\, \ell}_0; U^J_j)\nn\\
		&= F^I_{j,x_i}(x\,'^{\, \ell}_0;U^I_j) 
		+\sum_{J\in\I^I_j} F^I_{j,u^J}(x\,'^{\, \ell}_0;U^I_j)
		F^J_i(x\,'^{\, \ell}_0;U^J_i).\label{syst1_int_condb}
	\end{align}
	Finally, it follows from \eq{3_or_more} that \eq{syst1_int_conda} results from \eq{syst1_int_condb} upon 
	renaming any dependent variable $u^K$ in \eq{syst1_int_condb} by $u^{1,K}$.
\end{proof}
From the inductive hypothesis we conclude that, for each $\ell=1,\dots,n-1$, system $\ell$
with data \eq{ell_data} possess a unique solution $u^{\ell,I}(x\,'^{\, \ell})$ defined on a full $(n-1)$-dimensional 
neighborhood of $0\in\RR^{n-1}_{x\,'^{\, \ell}}$.

Before we continue we make the following observation that will be important later.
It is a consequence of the fact that the data assigned for system $\ell$ are independent 
of $\ell$, cf.\ \eq{ell_data}.

\begin{claim}\label{key_id_claim}
	Assume that $I\in\I$ contains two distinct indices $k$ and $\ell$.
	Then the solutions $u^{k,I}$ and $u^{\ell,I}$ of systems $k$ and $\ell$,
	respectively, satisfy
	\beq\label{key_id}
		u^{k,I}(x\,'^{\, k})|_{\{x_\ell=0\}}=u^{\ell,I}(x\,'^{\, \ell})|_{\{x_k=0\}}.
	\eeq
\end{claim}
\begin{proof}
	Assume $k<\ell$. If $I$ is the 2-index $(k,\ell)$, then $I'^k=(\ell)$, $I'^\ell=(k)$,
	and the data assignment \eq{ell_data} gives 
	\begin{align*}
		u^{k,I}(x\,'^{\, k})|_{\{x_\ell=0\}} 
		&= u^{k,I}(x\,'^{\, k})|_{\{x_{I'^k}=0\}}\\
		&= \bar u^{I}(x_{I^c})\\
		&= u^{\ell,I}(x\,'^{\, \ell})|_{\{x_{I'^\ell}=0\}}= u^{\ell,I}(x\,'^{\, \ell})|_{\{x_k=0\}}, 
	\end{align*}
	which verifies \eq{key_id} in this case. Next assume $|I|\geq 3$, $k<\ell<n$, 
	and set
	\[\I_{k,\ell}:=\{I\in\I\,|\, k,\ell\in I\neq(k,\ell)\}.\]
	Then, for each $I\in\I_{k,\ell}$, system $k$ contains the equations
	\beq\label{part_k_syst}
		u^{k,I}_{x_i}(x\,'^{\, k})=F^I_i(x\,'^{\, k}_0;U^{k,I}_i(x\,'^{\, k}))\qquad i\in I\minus \{k,\ell\},
	\eeq
	and system $\ell$ contains the equations
	\beq\label{part_ell_syst}
		u^{\ell,I}_{x_i}(x\,'^{\, \ell})=F^I_i(x\,'^{\, \ell}_0;U^{\ell,I}_i(x\,'^{\, \ell}))\qquad i\in I\minus \{k,\ell\}.
	\eeq
	We now let
	\[\tilde x:=x\,'^{\, k,l}=(x_1,\dots,x_{k-1},x_{k+1},\dots,x_{\ell-1},x_{\ell+1},\dots,x_n),\]
	and define the functions
	\[v^I(\tilde x):=u^{k,I}(x\,'^{\, k})|_{\{x_\ell=0\}}
	\qquad\text{and}\qquad w^I(\tilde x):=u^{\ell,I}(x\,'^{\, \ell})|_{\{x_k=0\}}.\]
	Letting 
	\[V^I_i:=\{v^J\,|\, J\supset I\minus i,\, J\neq (k,\ell)\},\]
	and 
	\[W^I_i:=\{w^J\,|\, J\supset I\minus i,\, J\neq (k,\ell)\},\]
	and restricting \eq{part_k_syst} and \eq{part_ell_syst} to $\{x_\ell=0\}$ 
	and $\{x_k=0\}$, respectively, we deduce that the sets of functions 
	\[\{v^I(\tilde x)\,|\,I\in\I_{k,\ell}\}\qquad\text{and}\qquad \{w^I(\tilde x)\,|\,I\in\I_{k,\ell}\}\]
	solve the following system-data pairs:
	\begin{align*}
		& v^{I}_{x_i}(\tilde x) =F^I_i(x\,'^{\, k\ell}_{00};[\bar u^{(k,\ell)}(\tilde x)],
		V^{I}_i(\tilde x))\qquad\text{for $i\in I\minus \{k,\ell\}$}\\
		& v^{I}(\tilde x)|_{\{x_{I'^{kl}}=0\}} = \bar u^I(x_{I^c})
	\end{align*}
	and 
	\begin{align*}
		& w^{I}_{x_i}(\tilde x) =F^I_i(x\,'^{\, k\ell}_{00};[\bar u^{(k,\ell)}(\tilde x)],
		W^{I}_i(\tilde x))\qquad\text{for $i\in I\minus \{k,\ell\}$}\\
		& w^{I}(\tilde x)|_{\{x_{I'^{kl}}=0\}} = \bar u^I(x_{I^c}),
	\end{align*}
	respectively. Here we are using the same notational convention as earlier:
	the square-bracketed terms are present only if $I$ is one of the 3-indices
	$(i,k,\ell)$, $(k,i,\ell)$ or $(k,\ell,i)$. 
	
	We note that the two systems above are identical, and we claim that they are 
	closed Darboux systems in $n-2$ independent variables. 
	The argument for this claim is entirely similar to that for Claim 
	\ref{claim:syst_ell} above; we omit the details.
	Since the systems above for the $v^I$ and the $w^I$ 
	have the same data, it follows from the uniqueness part of Darboux's 
	theorem (for $n-2$ independent variables) that $v^I\equiv w^I$, i.e.\ 
	\eq{key_id} holds.
\end{proof}

\subsection{System $n$} 
To generate our candidate for the solution of the original system we 
proceed to define system $n$, which will be a determined system
in all $n$ independent variables $x=(x_1,\dots,x_n)$. We denote its 
unknowns by $\hat u^I$ where $I$ now ranges over all of $\I$. 
For each $I\in\I$, 
and with $i:=\min I$, we copy the $u_{x_i}$-equation from the original 
system \eq{gen_form_3}
and then rename any unknown $u^K$ appearing in it as $\hat u^K$:
\beq\label{n_syst}
	\hat u^I_{x_i}(x)=F^I_i(x;\hat U^I_i(x))\qquad\text{with $i=\min I$,}
\eeq
where 
\[\hat U^I_i:=\{\hat u^K\,|\, K\in\I,\, K\supset I\minus i\}.\]
The data for system $n$ are prescribed as follows:
\begin{enumerate}
	\item if $I=(i)$ for some $i\in\{1,\dots,n\}$, then 
	\beq\label{n_th_syst_data_a}
		\hat u^{(i)}(x)|_{\{x_i=0\}}:=\bar u^{(i)}(x\,'^{\, i});  
	\eeq
	\item if $|I|>1$ and $\min I=i$, then
        \beq\label{n_th_syst_data_b}
	        	\hat u^{I}(x)|_{\{x_i=0\}}:= u^{i,I}(x\,'^{\, i}).
        \eeq
\end{enumerate}
Note that we use the solutions for systems 1 through $n-1$ to 
assign the data for system $n$.
Since \eq{n_syst} contains exactly one equation for each of the 
unknowns $\hat u^I$, where $I$ ranges over $\I$, it follows that system $n$ is a closed
and determined system. Also, the data are of the form required by Darboux's 
first theorem. Hence, according to Theorem \ref{darb_1}, system $n$
with the data in \eq{n_th_syst_data_a}-\eq{n_th_syst_data_b} has 
a unique, local solution $\hat U$ defined in a full $n$-dimensional neighborhood 
of $0\in\RR^{n}_{x}$.

It remains to verify that the solution $\hat U$ of system $n$ in fact solves the
original problem \eq{gen_form_3} \& \eq{gen_data}. For this, consider first the
data requirements, which are straightforward to verify. Indeed, if $|I|=1$, say $I=(i)$, 
then according to \eq{n_th_syst_data_a} we have
\[\hat u^I(x)|_{\{x_I=0\}}\equiv \hat u^{(i)}(x)|_{\{x_i=0\}}
=\bar u^{(i)}(x\,'^{\, i})\equiv \bar u^I(x_{I^c}),\]
and if $|I|>1$, with $\min I=i$, then according to 
\eq{n_th_syst_data_b} and \eq{ell_data} we have
\[\hat u^I(x)|_{\{x_I=0\}}
\equiv\big(\hat u^I(x)|_{\{x_i=0\}}\big)|_{\{x_{I^{'i}}=0\}}
=u^{i,I}(x\,'^{\, i})|_{\{x_{I^{'i}}=0\}}
= \bar u^I(x_{I^c}).\]
This shows that $\hat U$ takes on the original data \eq{gen_data}.
	
To show that $\hat U$ solves all equations in \eq{gen_form_3}, we start by noting 
that, by construction, the solutions $\hat u^I$ ($I\in\I$) of the $n$th 
system solve all equations in \eq{gen_form_3} with $i=\min I$.
The remaining equations in the original system 
are those equations in \eq{gen_form_3} where an unknown $u^I$ is 
differentiated with respect to an $x_i$ where $i\in I$ and $i>\min I$. 

To verify these remaining equations we shall employ the same strategy 
used by Darboux. For this we define, for each $I\in\I$, the functions
\beq\label{Delta}
	\Delta^I_i(x):=\hat u^I_{x_i}(x)-F^I_i(x;\hat U^I_i(x))
	\qquad\text{for $i\in I$},
\eeq
and set
\beq\label{D}
	\mathcal D:=\{\Delta^I_i\,|\, i\in I\in\I,\, i>\min I\}.
\eeq
Note that $\Delta_i^I\equiv 0$ whenever $i=\min I$. 
We will show, through an application of Theorem \ref{darb_1}, that 
also all $\Delta^I_i$ in $\mathcal D$ vanish identically 
on a full neighborhood of $0\in\RR^n$. By definition of $\Delta^I_i$ 
this establishes that $\hat U$ solves all the equations in the original system 
\eq{gen_form_3}.
Verifying that all $\Delta^I_i$ in $\mathcal D$ vanish identically near $x=0$, 
will be accomplished in two steps by showing that:
\begin{enumerate}
	\item[(A)] whenever $\Delta^I_i\in\mathcal D$, 
	then the partial derivative
	\[\partial_{x_\ell}\Delta^I_i\qquad\text{where $\ell=\min I,$}\]
	is given as a linear combination (with variable coefficients) of other
	functions $\Delta^K_k$ in $\mathcal D$;
	\item[(B)] each $\Delta^I_i\in \mathcal D$ vanishes identically 
	along $\{x_\ell=0\}$ for $\ell=\min I$.
\end{enumerate}
It follows from (A) and (B) that the functions $\Delta^I_i$ in $\mathcal D$ 
satisfy a determined system of linear PDEs of the 
type covered by Darboux's first theorem, and with vanishing data.  
According to the uniqueness part of that theorem, it follows 
that all $\Delta^I_i$ vanish identically near $0\in \RR^n$, which
is the desired conclusion.

\subsubsection{Proof of (A)}
We now fix $I\in\I$ and $i\in I$ such that $\ell:=\min I<i$.  
Then, as $\ell=\min I$ and $\hat u^I$ is part of the solution to the $n$th 
system, we obtain:
\begin{align}
	\partial_{x_\ell}\Delta^I_i(x) 
	&= \left[\hat u^I_{x_\ell}(x)\right]_{x_i}
	-\left[F^I_i(x;\hat U^I_i(x))\right]_{x_\ell}\nn\\
	&= \left[F^I_\ell(x;\hat U^I_\ell(x))\right]_{x_i}
	-\left[F^I_i(x;\hat U^I_i(x))\right]_{x_\ell}\nn\\
	&= F^I_{\ell,x_i}(x;\hat U^I_\ell(x)) 
	+ \sum_{K\supset I\minus\ell} 
	F^I_{\ell,u^K}(x;\hat U^I_\ell(x)) \hat u^K_{x_i}(x)\nn\\
	&\quad - F^I_{i,x_\ell}(x;\hat U^I_i(x))
	- \sum_{K\supset I\minus i} 
	F^I_{i,u^K}(x;\hat U^I_i(x)) \hat u^K_{x_\ell}(x)\nn\\
	&=\sum_{K\supset I\minus\ell} 
	 A^I_{\ell,K}(x)\Delta^K_i(x)
	-\sum_{K\supset I\minus i} 
	A^I_{i,K}(x) \Delta^K_\ell(x),\label{d_sum}
\end{align}
where in the last step we have used the integrability conditions 
\eq{int_cond2}, and also introduced the coefficient functions
\[A^I_{\ell,K}(x):=F^I_{\ell,u^K}(x;\hat U^I_\ell(x)).\]
Note that, in the first sum in \eq{d_sum}, $K\supset  I\minus\ell\ni i$, such 
that $\Delta^K_i$ is defined by \eq{Delta}; ditto for the second sum in 
\eq{d_sum}, with $i$ and $\ell$ interchanged.
Finally, since the $\hat u^K$ solve the $n$th system, we have that:
\begin{itemize}
	\item those $\Delta^K_i$ in the first sum in \eq{d_sum} 
	for which $\min K=i$ vanish identically, and
	\item those $\Delta^K_\ell$ in the second sum in \eq{d_sum} 
	for which $\min K=\ell$ vanish identically.
\end{itemize}
This shows that whenever $\Delta^I_i\in \mathcal D$ and $\ell=\min I$, then
$\partial_{x_\ell}\Delta^I_i$ is a linear combination of functions 
from $\mathcal D$.  \hfill{q.e.d.(A)}

\subsubsection{Proof of (B)}
Fix $I$ and $i$ as above, i.e.\ $\ell:=\min I<i\in I$, and calculate:
\begin{align}
	\Delta^I_i(x)|_{\{x_\ell=0\}} &=\hat u^I_{x_i}(x\,'^{\, \ell}_0)
	-F^I_i(x\,'^{\, \ell}_0;\hat U^I_i(x\,'^{\, \ell}_0))\nn\\
	&=\del_{x_i}\left[\hat u^I(x\,'^{\, \ell}_0) \right]
	-F^I_i(x\,'^{\, \ell}_0;[\bar u^{(\ell)}(x\,'^{\, \ell})],\hat U^{\ell,I}_i(x\,'^{\, \ell}_0)),
	\label{delta}
\end{align}
where we have set 
\[\hat U^{\ell,I}_i:=\{\hat u^K\,|\, K\in\I^{\ell,I}_i\},\]
with $\I^{\ell,I}_i$ given in \eq{I^ell^I_i} (recall that $i\in I\in\I_\ell$). We then split
$\I^{\ell,I}_i$ into two parts as follows:
\[\I^{\ell,I}_{i,a}:=\{K\in \I^{\ell,I}_i\,|\, \min K=\ell\}\]
and 
\[\I^{\ell,I}_{i,b}:=\{K\in \I^{\ell,I}_i\,|\, \min K<\ell\}.\]
According to the data assignment \eq{n_th_syst_data_b} for system $n$, we have
\beq\label{key}
	\hat u^K(x\,'^{\, \ell}_0)=u^{\ell,K}(x\,'^{\, \ell})\qquad\text{whenever $K\in \I^{\ell,I}_{i,a}$.}
\eeq
We now have the following claim:
\begin{claim}\label{key_claim}
	With $I$, $\ell$, and $i$ as above, we have
	\beq\label{key_key}
		\hat u^K(x\,'^{\, \ell}_0)=u^{\ell,K}(x\,'^{\, \ell})\qquad\text{whenever $K\in \I^{\ell,I}_{i,b}$ as well.}
	\eeq
\end{claim}
\noindent Assuming this result for now (proved below), we get from \eq{delta} that
\[\Delta^I_i(x)|_{\{x_\ell=0\}} =\del_{x_i}\left[u^{\ell,I}(x\,'^{\, \ell}) \right]
-F^I_i(x\,'^{\, \ell}_0;[\bar u^\ell(x\,'^{\, (\ell)})],U^{\ell,I}_i(x\,'^{\, \ell})).\]
This last expression vanishes identically since $u^{\ell,I}$ is part of the 
solution to system $\ell$, verifying the claim in (B).

It only remains to argue for Claim \ref{key_claim}. With $I$, $\ell$, and $i$ as above, 
we introduce the following notation: whenever $K\in \I^{\ell,I}_{i,b}$,
we set
\[v^K(x\,'^{\, \ell}):=u^{\ell,K}(x\,'^{\, \ell})\qquad\text{and}\qquad w^K(x\,'^{\, \ell}):=\hat u^K(x\,'^{\, \ell}_0).\]
Claim \ref{key_claim} amounts to the statement that $v^K(x\,'^{\, \ell})=w^K(x\,'^{\, \ell})$.
We shall establish that the functions $v^K$ and $w^K$ satisfy the same 
determined system with the same data. The uniqueness part of Darboux's first theorem 
will then yield the conclusion.

We have, with $K\in \I^{\ell,I}_{i,b}$ and $k:=\min K$, that
\begin{align}
	v^K_{x_k}(x\,'^{\, \ell}) = u^{\ell,K}_{x_k}(x\,'^{\, \ell}) 
	&= F_k^K(x\,'^{\, \ell}_0;[\bar u^{(\ell)}(x\,'^{\, \ell})],U^{\ell,K}_k(x\,'^{\, \ell}))\nn\\
	&= F_k^K(x\,'^{\, \ell}_0;[\bar u^{(\ell)}(x\,'^{\, \ell})],U^{\ell,K}_{k,a}(x\,'^{\, \ell}),U^{\ell,K}_{k,b}(x\,'^{\, \ell}))\nn\\
	&=: G_k^K(x\,'^{\, \ell}_0;V^{\ell,K}_{k,b}(x\,'^{\, \ell})),\label{v^K}
\end{align}
where we regard $U^{\ell,K}_{k,a}(x\,'^{\, \ell})$ as a set of known functions, and we have set
$V^{\ell,K}_{k,b}=U^{\ell,K}_{k,b}$. 
We note that the resulting 
system of equations contains exactly one equation for each $K\in \I^{\ell,I}_{i,b}$. The data for this system
is given as
\[v^K(x\,'^{\, \ell})|_{\{x_k=0\}}=u^{\ell,K}(x\,'^{\, \ell})|_{\{x_k=0\}}.\]
Similarly,
\begin{align}
	w^K_{x_k}(x\,'^{\, \ell}) = \hat u^K_{x_k}(x\,'^{\, \ell}_0) 
	&= F_k^K(x\,'^{\, \ell}_0;[\bar u^{(\ell)}(x\,'^{\, \ell})],\hat U^{\ell,K}_k(x\,'^{\, \ell}_0))\nn\\
	&= F_k^K(x\,'^{\, \ell}_0;[\bar u^{(\ell)}(x\,'^{\, \ell})],\hat U^{\ell,K}_{k,a}(x\,'^{\, \ell}),\hat U^{\ell,K}_{k,b}(x\,'^{\, \ell})),\nn
\end{align}
where we have set
\[\hat U^{\ell,K}_{k,a}:=\{\hat u^K\,|\, K\in\I^{\ell,I}_{i,a}\}\qquad\text{and}\qquad
\hat U^{\ell,K}_{k,b}:=\{\hat u^K\,|\, K\in\I^{\ell,I}_{i,b}\}.\]
According to \eq{key} we have 
\[\hat U^{\ell,K}_{k,a}(x\,'^{\, \ell})\equiv U^{\ell,K}_{k,a}(x\,'^{\, \ell}),\]
so that 
\begin{align}
	w^K_{x_k}(x\,'^{\, \ell}) = \hat u^K_{x_k}(x\,'^{\, \ell}_0) 
	&= F_k^K(x\,'^{\, \ell}_0;[\bar u^{(\ell)}(x\,'^{\, \ell})],U^{\ell,K}_{k,a}(x\,'^{\, \ell}),\hat U^{\ell,K}_{k,b}(x\,'^{\, \ell}))\nn\\
	&\equiv G_k^K(x\,'^{\, \ell}_0;W^{\ell,K}_{k,b}(x\,'^{\, \ell})),\label{w^K}
\end{align}
where $G^K_k$ was defined in \eq{v^K} and we have set $W^{\ell,K}_{k,b}=\hat U^{\ell,K}_{k,b}$. 
Again, these equations for the $w^K$, $K\in \I^{\ell,I}_{k,b}$  form a determined system.
Finally, the data for $w^K$ are given as follows:
\[w^K(x\,'^{\, \ell})|_{\{x_k=0\}}=\hat u^{K}(x\,'^{\, \ell}_0)|_{\{x_k=0\}}
\equiv \hat u^{K}(x\,'^{\, k}_0)|_{\{x_\ell=0\}}=u^{k,K}(x\,'^{\, k})|_{\{x_\ell=0\}},\]
where we have used that $k=\min K$, together with the data assignment \eq{n_th_syst_data_b}
for $\hat u^K$. Finally, we apply Claim \ref{key_id_claim} to conclude that the data for $w^K$ are given by
\[w^K(x\,'^{\, \ell})|_{\{x_k=0\}}=u^{\ell,K}(x\,'^{\, \ell})|_{\{x_k=0\}}.\]
Thus, the $v^K$ and the $w^K$ solve the same determined system with the same data,
and the uniqueness part of Darboux's first theorem implies that $v^K\equiv w^K$ for each $K\in \I^{\ell,I}_{i,b}$.
\hfill{q.e.d.(B)}

The uniqueness part of Darboux's theorem follows from the inductive construction above.
This concludes the proof of Theorem \ref{d_thm}.

%
%


\end{document}